\documentclass[12pt]{amsart}

\usepackage{amsmath}
\usepackage{amssymb}
\usepackage{amsthm}
\usepackage[margin=1in, top=1.5in]{geometry}

\usepackage[slantedGreek]{mathptmx}
\usepackage{url}
\usepackage{bbm}

\newcommand{\id}{\mathbbm{1}}
\newcommand{\supp}{\mbox{\rm supp}\; }
\newcommand{\grad}{\nabla{}}
\renewcommand{\div}{\mbox{\rm div}}

\newtheorem{theorem}{Theorem}
\newtheorem{lemma}{Lemma}

\newtheorem*{theorem*}{Theorem}

\newcommand{\e}{\varepsilon}

\title{Gibb's minimization principle for approximate solutions of scalar  conservation laws}
\author{Misha Perepelitsa}

\thanks{Email: misha@math.uh.edu}

\date{\today}
\address{Misha Perepelitsa\\
misha@math.uh.edu\\
University of Houston\\
PGH 631\\
4800 Calhoun Rd. \\
Houston, TX\\
USA}

\begin{document}
\begin{abstract}
In this work we study variational properties of approximate solutions of scalar conservations laws. Solutions of this type are described by a kinetic equation which is similar to the kinetic representation of admissible weak solutions due to Lions-Perthame-Tadmor\cite{LPT}, but also retain small scale non-equilibrium behavior. We show that approximate solutions can be obtained from a BGK-type equation with equilibrium densities satisfying  Gibb's entropy minimization principle.
\end{abstract}

\maketitle

\begin{section}{Introduction}
\subsection{Motivation}
We consider a Cauchy problem for a scalar conservation law
\begin{equation}
\label{SCL}
\begin{array}{ll}
\partial_t \rho{}+{}\div_x A(\rho){}={}0,& (x,t)\in\mathbb{R}^{d+1}_+,\\
\rho(x,0){}={}\rho_0(x), & x\in\mathbb{R}^d,
\end{array}
\end{equation}
where $A\;:\;\mathbb{R}\to\mathbb{R}^{d}$ is a Lipschitz continuous function.
For initial data
\[
\rho_0\in L^\infty(\mathbb{R}^d)\cap L^1(\mathbb{R}^d),
\]
the problem is uniquely solvable in the class of admissible (entropy) solutions, as was established in \cite{Kruzhkov}. 
When an admissible solution $\rho(x,t)$ is represented by a kinetic density as
\[
\rho(x,t){}={}\int f(x,t,v)\,dv,
\]
with
\begin{equation}
\label{Eq_density}
f(x,t,v){}={}
\left\{
\begin{array}{rl}
\id_{[0,\rho(x,t)]}, & \rho(x,t)\geq0\\
-\id_{[\rho(x,t),0]}, & \rho(x,t)<0
\end{array}
\right.,
\end{equation}
then $f$ is a weak solution of a kinetic equation
\begin{equation}
\label{kinetic_eq_1}
\partial_t f{}+{}A'(v)\cdot\grad_x f{}={}-\partial_v m,\quad \mathcal{D}'(\mathbb{R}^{2d+1}_+),
\end{equation}
where $m$ is non-negative Radon measure on $\mathbb{R}^{2d+1}_+.$ Conversely, any solution of \eqref{kinetic_eq_1} constrained by condition \eqref{Eq_density} for some $\rho(x,t)$ defines an admissible weak solution of conservation law in \eqref{SCL}, see \cite{LPT}. Kinetic methods for obtaining admissible solutions originate in works \cite{Brenier, GM}. References \cite{BB, BV, Bouchut, BGP, LPS, LPT2, Perthame} is an short list  of some representative results of the kinetic approach to solving systems of quisilinear PDEs.

Given a kinetic denisty $f,$ with $\rho=\displaystyle{\int f\,dv},$ we will denote an equilibrium density in \eqref{Eq_density} by $\Pi^{eq}_f.$

A class of approximate weak solutions of scalar conservations laws and equations of gas dynamics was introduced in \cite{M1, M2}. An approximate solution $\rho(x,t)$ of \eqref{SCL} is characterized by the following properties. For any $\e>0$ there is $\rho$ such that
\begin{enumerate}
\item[P1:] $\rho$ is a weak solution of the equation
\[
\partial_t\rho {}+{}\div_x \rho\left( A(\rho)/\rho +O(\e)\right){}={}0,
\]
where $O(\e)$ is function of $(x,t)$ such that
\[
|O(\e)|{}\leq{} C\e,\quad \mbox{a.e. } (x,t)\in \mathbb{R}^{d+1}_+;
\]
\item[P2:] $\rho$ has a kinetic representation
\[
\rho(x,t){}={}\int f(x,t,v)\,dv,
\]
with $f$ solving a kinetic equation
\[
\partial_t f{}+{}A'(v)\cdot\grad_x f{}={}-\partial_v m,
\]
where both $m$ and $\partial_v m$ are Radon measures  $\mathbb{R}^{2d+1}_+,$ $m$ is being non-negative;
\item[P3:] the kinetic density $f$ deviates slightly from the equilibrium density:
\begin{equation}
\label{Deviation}
D(f){}={}\frac{\displaystyle{\int v(f-\Pi^{eq}_f)\,dv}}{\displaystyle{\int vf\,dv}} {}\leq \e,\quad \mbox{a.e. } (x,t)\in\mathbb{R}^{d+1}_+;
\end{equation}
\item[P4:] there is a parametrized, unit mass, measure $\mu_{x,t}$ on $\mathbb{R}^d_v,$ such that $\mu_{x,t}$ is a measure-valued solution of the equation in \eqref{SCL}:
\[
\partial_t\langle \rho,\mu_{x,t}\rangle{}+{}\div_x\langle A(\rho),\mu_{x,t}\rangle{}={}0,
\]
and $\mu_{x,t}$ is close to a delta mass concentrated at $\rho(x,t):$
\[
\mu_{x,t}{}={}\delta(v-\rho(x,t)){}+{}\mu^\e_{x,t},
\]
with 
\[
\mbox{mass}|\mu^\e_{x,t}|{}\leq{} C\e,\quad \mbox{a.e. } (x,t)\in\mathbb{R}^{d+1}_+.
\]
\end{enumerate}
An example of an approximate solution corresponding to stationary shock data for Burger's equation was constructed in \cite{M1}. The solution has a sharp interface of discontinuity (shock) which is $\e$--close to a classical shock, but it also contains $\e$--small rarefaction waves that interact with the shock and travel through it.
It is unlikely that conditions P1--P4 determine approximate solutions in a unique fashion: there is large ``amount of indeterminacy'' in condition P1. However, the method that is used to construct them in \cite{M1} (described below), in dimension 1, results in  approximate solutions that coincide with smooth solutions of \eqref{SCL}, and for some initial data, coincide with shocks of \eqref{SCL} as well. In fact, it is possible to show that a sequence of approximate solutions $\{\rho^\e\}$ with $\e\to0,$ accumulates on an admissible solution of \eqref{SCL}.

In \cite{M1} approximate solutions are constructed by taking zero relaxation limit of a family of solutions of a BGK model
\begin{equation}
\label{BGK1}
\partial_t f{}+{}A'(v)\cdot\grad_x f{}={}\frac{\Pi^{eq}_f-f}{h}\id_{\{(x,t)\,:\,D(f(x,t,\cdot))>\e\}},
\end{equation}
where the deviation $D(f)$ is defined in \eqref{Deviation}.

Note that at the points $(x,t)$ where $D(f)\leq\e,$ equation \eqref{BGK1} is a linear transport equation, which results in small (but non-vanishing) regularization of $f$ due to dispersion (mixing). This regularization is expressed in property P2 above, by the condition that $\partial_v m$ is a Radon measure.

A limiting point $f{}={}\lim f^h$ is located near the set of equilibrium densities, as expressed by the condition $D(f)\leq \e,$ a.e. $(x,t).$ Thus, approximate solutions retain some small scale non-equilibrium features of kinetic equation $\eqref{BGK1}.$

This framework applies equally well to systems of conservation laws that have a kinetic representation, see \cite{M2} for an example of equations of isentropic gas dynamics.

In this paper we show that approximate solutions can obtained in a zero relaxation limit of BGK model
\begin{equation}
\label{BGK2}
\partial_t f{}+{}A'(v)\cdot\grad_x f{}={}\frac{\Pi^\e_f-f}{h},
\end{equation}
where $\Pi^\e_f$ is a solution of minimization problem
\begin{equation}
\label{prob:min}
\min\left\{ \int\eta^\e(v)f\,dv\right\},
\end{equation}
constrained by conditions 
\[
\int f\,dv{}={}const.,\quad f(v)\in[0,1],
\]
where $\eta^\e(v)$ is a piece-wise constant approximation of entropy $\eta(v){}={}v$ which defines  Gibb's entropy $S(f){}={}\displaystyle{\int v f\,dv}.$ The minimizer of the later is the equilibrium kinetic densities $\Pi^{eq}_f,$ as in \eqref{Eq_density}.  The restriction of $f$ to have non-negative values can be made by considering only non-negative solutions $\rho(x,t),$ which can be assumed without the loss of generality. This approach formally puts the kinetic equation for approximate solutions \eqref{BGK2} into a classical framework of kinetic equations in gas dynamics, in which the equilibrium density is a minimizer of an entropy, subject to prescribed moments. The important difference is that minimizers of \eqref{prob:min} are not unique. In fact, we use this non-uniqueness to select a minimizer that is regularized by dispersion at $\e$ scales, see lemma \ref{lemma:2}.

In our variational approach $\e$ has  different interpretation. Whereas in \eqref{BGK1} it was a non-dimensional quantity measuring relative deviation of the entropy, here, we measure the deviation of $f$ from the equilibrium by
\[
D(f){}={}\frac{\displaystyle{\int v(f-\Pi^{eq}_f)\,dv}}{\displaystyle{\int f\,dv}}.
\]
Thus, $\e$ has the dimension of the kinetic variable $v.$

Our main result established an approximate solution $\rho$ that verifies properties P1--P4, with the above $D(f).$ In addition, we improve condition P4, by showing that  a measure-valued representation of $\rho(x,t){}={}\langle \rho,\mu_{x,t}\rangle$ with the measure $\mu_{x,t}$ is supported near $v{}={}\rho(x,t):$
\[
\mbox{diam}(\supp \mu_{x,t}){}\leq{}C\e,\quad \mbox{a.e. } (x,t).
\]

\subsection{Main results}

In the rest of the paper we  always assume that $A\in C^2(\mathbb{R})^d$ and verifies a non-degeneracy condition:
\begin{equation}
\label{main:nondegeneracy}
\forall\, \sigma \in \mathbb{S}^{d-1},\quad \forall \xi\in\mathbb{R},\quad \left|\left\{ v\in (-\|\rho_0\|_{L^\infty},\|\rho_0\|_{L^\infty})\;:\; A'(v)\cdot\sigma{}={}\xi\right\}\right|{}={}0,
\end{equation}
where $\mathbb{S}^{d-1}$ is the unit sphere in $\mathbb{R}^d.$

\begin{theorem} 
\label{th:1}
Let $\rho_0\in L^\infty(\mathbb{R}^d)\cap L^1(\mathbb{R}^d).$  There are functions $\rho=\rho(x,t)$ and $f=f(x,t,v),$ with
\[
\rho{}={}\int f\,dv,\quad \mbox{a.e. }(x,t),
\]
and $C>0,$ that verify the following properties.
\begin{enumerate}

\item $\rho\in L^\infty(\mathbb{R}^{d+1}_+)\cap C([0,+\infty);W^{-1,p}_{loc}(\mathbb{R}^{d+1}_+)),$ for any $p\in(1,+\infty),$ and verifies in the weak sense the initial condition in \eqref{SCL}. $\rho$ is a weak solution of the equation
\begin{equation}
\label{part:1}
\partial_t \rho{}+{}\div_x\rho\left( A(\rho)/\rho + \tilde{A}(x,t)\right){}={}0,
\end{equation}
for some functions $\tilde{A}(x,t)$ with 
\[
\|\tilde{A}\|_{L^\infty(\mathbb{R}^{d+1}_+)}{}\leq C\e;
\]

\item for every convex function  $\eta$ on $[0,M],$
\[
\partial_t \int \eta' f\,dv{}+{}\div_x \int \eta' A'f\,dv\leq 0,\quad \mathcal{D}'(\mathbb{R}^+_{x,t});
\]

\item there is $m$ -- a non-negative Radon measure on $\mathbb{R}^+_{x,t,v}$ such that $\partial_v m$ is signed Radon measure and $f$ is a distributional solution of the equation
\[
\partial_t f{}+{}A'\cdot\grad_x f{}={}-\partial_v m.
\]

\item for  a.e. $(x,t),$
\[
\int v(f-\Pi^{eq}_f){}\leq{} 4\e\int v f\,dv;
\]

\item there is a parametrized, unit mass, measure $\mu_{x,t}$ on $\mathbb{R}^d_v,$ such that $\mu_{x,t}$ is a measure-valued solution of the equation in \eqref{SCL}:
\[
\partial_t\langle \rho,\mu_{x,t}\rangle{}+{}\div_x\langle A(\rho),\mu_{x,t}\rangle{}={}0,
\]
and $\mu_{x,t}$ is close to a delta mass concentrated at $\rho(x,t):$
\[
\mu_{x,t}{}={}\delta(v-\rho(x,t)){}+{}\mu^\e_{x,t},
\]
with 
\[
\mbox{\rm mass}|\mu^\e_{x,t}|,\,\mbox{\rm diam}(\supp \mu^\e_{x,t}){}\leq{} C\e,\quad \mbox{a.e. } (x,t)\in\mathbb{R}^{d+1}_+.
\]
\end{enumerate}
\end{theorem}

\subsection{Proof of theorem \ref{th:1}}
\begin{proof}
Assume that $\rho_0$ is non-negative and denote by $M{}={}1+\mbox{\rm ess }\sup \rho_0.$ 
Define a piecewise constant function $\eta_\e$ as
\[
\eta_\e(v){}={}k,\quad v\in[(k-1)\e,k\e),\,k=1..\lceil M/\e\rceil.
\]
For a non-negative constant $\rho\in[0,\mbox{\rm ess }\sup \rho_0]$ consider a minimization problem
\[
\min\left\{ \int \eta_\e(v)f(v)\,dv\,:\, f(v)\in[0,1],\,\int f\,dv{}={}\rho\right\}.
\]
Here and below $\displaystyle{\int f\,dv{}={}\int_0^Mf\,dv}.$

\begin{lemma}
\label{lemma:1} Let $
N{}={}\lfloor \rho/\e\rfloor.$
The minimum of the above problem equals
\[
\left\{
\begin{array}{ll}
\e\sum_{k=0}^{N-1}k+ \e N(\rho-N\e), & N\geq 1,\\
0, & N=0.
\end{array}
\right.
\]
It is achieved on the minimizers
\[
f_{min}(v){}={}\id_{[0,N\e]}(v){}+{}\tilde{f}(v),
\]
where $\tilde{f}$ is an arbitrary function verifying conditions:
\begin{align*}
&\tilde{f}(v)\in [0,1],\quad \forall v\in[0,M];\\
&\supp\tilde{f}\subset [N\e,(N+1)\e];\\
&\int\tilde{f}\,dv{}={}\rho - N\e.
\end{align*}
\end{lemma}
\begin{proof}
$\eta_\e(v)$ is a non-decreasing function. To minimize the functional $\displaystyle{\int \eta_\e f\,dv}$ we need to pick $f$ that has all its mass as close to $v=0$ as possible, and is less than or equal $1.$ This leads to the statement of the lemma.
\end{proof}
In the next lemma we show that the decrease of the entropy controls $L^1$ distance between function $f$ and a certain minimizer $f_{min}.$
\begin{lemma}
\label{lemma:2}
 Let $f$ be any function with values in $[0,1],$ with mass equal to $\rho.$ If $f_{min}$ is a minimizer from the last lemma, and
\[
f_{min}(v)\geq f(v),\quad v\in[N\e,(N+1)\e],
\]
Then, for $\e\leq1,$
\[
\int |f-f_{min}|\,dv \leq \frac{3}{\e}\int \eta_\e(f-f_{min})\,dv.
\]
\end{lemma}
\begin{proof} Consider first the case $N=0,$ or $N{}={}\rho/\e{}={}1.$ Under these conditions
\[
\int \eta_\e f_{min}\,dv{}={}0.
\]
Then,
\begin{multline*}
\int |f-f_{min}|\,dv{}={}\int_0^\e f_{min}-f\,dv{}+{}\int_\e^Mf\,dv
{}={}2\int_\e^Mf\,dv\\
\leq \frac{2}{\e}\int_\e^M\eta_\e f\,dv{}={}\frac{2}{\e}\int_0^M\eta_\e f\,dv{}={}\frac{2}{\e}\int_0^M\eta_\e(f-f_{min}).
\end{multline*}
Suppose now $N>1.$
\begin{multline*}
\int |f-f_{min}|\,dv{}={}\int_0^{N\e}(1-f)\,dv{}+{}\int_{N\e}^{(N+1)\e}(f_{min}-f)\,dv{}+{}\int_{(N+1)\e}^Mf\,dv \\
{}\leq {}\int_{N\e}^{(N+1)\e}(f_{min}-f)\,dv{}+{}2\int_{(N+1)\e}^Mf\,dv\\
{}\leq \frac{1}{\e}\left( \int_{N\e}^{(N+1)\e}\eta_\e(f_{min}-f)\,dv{}+{}2\int_{(N+1)\e}^M\eta_\e f\,dv\right){}\leq{}\frac{3}{\e}\int \eta_\e(f-f_{min})\,dv.
\end{multline*}

\end{proof}

A particular minimizer that verifies the conditions of the last lemma will be denoted by
\[
\Pi^\e_f(v){}={}\id_{[0, N\e+v_0]}(v){}+{}\id_{[N\e+v_0,(N+1)\e]}(v)f(v),
\]
where number $v_0\in[0,\e]$ equals
\[
v_0{}={}\max\left\{ 0, \int_0^{N\e}+\int_{(N+1)\e}^M f\,dv-N\e\right\}.
\]
\begin{lemma} Let $\eta$ be a convex function on $[0,M].$ 
Then,
\[
\int \eta(v) (f(v)-\Pi^\e_f(v))\,dv\geq 0.
\]
\end{lemma}
\begin{proof}
Restricted to the compliment of $[v_0,(N+1)\e],$ function
$\Pi^\e_f$ coincides with the equilibrium density of $f,$ restricted to the same set. For an equilibrium density $\Pi^{eq}_f$ the inequality is a well-know fact, shown for example in \cite{Brenier}. Since $f$ and $\Pi^\e_f$ coincide on $[v_0,(N+1)\e],$ the inequality follows.
\end{proof}

We consider the Cauchy problem 
\begin{align}
\label{BGK:1}& \partial_t f{}+{}A'(v)\cdot\grad_x f{}={}\frac{\Pi^\e_f-f}{h},\quad (x,t,v)\in \mathbb{R}^d\times\mathbb{R}_+\times[0,M],\\
&\label{BGK:2} f(x,0,v){}={}f_0(x,v), \quad (x,v)\in\mathbb{R}^d\times[0,M].
\end{align}

The proof of the next theorem can deduced by repeating the arguments of a result of \cite{BrenierCorias}, or theorem 4.7 of \cite{M1}, that apply to the same problem with $\Pi^{eq}_f,$ instead of $\Pi^\e_f$ on the right-hand side of the equation \eqref{BGK:1}. We omit the proof.
\begin{theorem}
Let $f_0\in L^\infty(\mathbb{R}^d\times[0,M]),$ with values in $[0,1]$ for a.e. $(x,v),$ with the support
\[
\supp f_0(x,\cdot)\subset [0,M],\quad \mbox{a.e. } x,
\]
and finite moments
\[
\iint (1+v)f_0(x,v)\,dxdv<+\infty.
\]
For any $h>0$ there is a weak solution of the problem \eqref{BGK:1}, \eqref{BGK:2}: for any $p\in[1,+\infty),$
\[
f\in L^\infty(\mathbb{R}^{d+1}_+\times[0,M])\cap L^\infty(0,+\infty; L^1(\mathbb{R}^d\times[0,M]))\cap C([0,+\infty); L^p_{loc,weak}(\mathbb{R}^d\times[0,M])),
\]
with the following properties: for all $t>0$ and a.e. $(x,v),$ $f(x,t,v)\in[0,1];$
\[
\supp f(x,t,\cdot)\subset[0,M],\quad \mbox{a. e. }(x,t);
\]
for all $t>0,$
\[
\iint(1+v) f(x,t,v)\,dxdv{}\leq{} \iint(1+v)f_0(x,v)\,dxdt.
\]
\end{theorem}

Solutions of a BGK model verify the following estimates.
\begin{lemma}[Entropy estimates] Let $f$ be a solution of \eqref{BGK:1}, \eqref{BGK:2} with properties listed in the above theorem. There exists $C>0,$ independent of $h,$ such that $\forall T>0,$
\begin{align}
\label{E:1}
& \sup_{[0,T]} \iint \eta_\e f(x,t,v)\,dvdx{}+{}\frac{1}{h}
\int_0^T\iint \eta_\e(f-\Pi^\e_f)\,dvdxdt{}\leq{} C,\\
\label{E:2}
&\frac{\e}{h}\int_0^T\iint |f-\Pi^\e_f|\,dvdxdt{}\leq{} C,\\
\label{E:3}
&\frac{\e}{h}\int_0^T\iint f\id_{[\rho(x,t)+\e,M]}(v)\,dvdxdt{}\leq{}C, \\
\label{E:4}
&\frac{\e}{h}\int_0^T\iint (1-f)\id_{[0,\max\{0,\rho(x,t)-\e\}]}(v)\,dvdxdt{}\leq{}C. 
\end{align}
\end{lemma}

Let $f_0(x,v)$ be the equilibrium density corresponding to initial data $\rho_0(x).$ Let $f^h$ be a sequence of solutions of \eqref{BGK:1}, \eqref{BGK:2} with such $f_0(x,t),$ and consider the compactness properties of $\{f^h\}$ as $t\to 0.$ Since $f^h$ are bounded in $L^\infty,$ and
the right-hand sides of \eqref{BGK:1} are bounded $L^1(\mathbb{R}^{d+1}_+\times[0,M]),$ due to estimate \eqref{E:2}, the compactness theorem of G\'{e}rard, see \cite{Gerard}, implies that
for any test function $\psi(v)$ the moments
\[
\left\{ \int \psi(v)f^h\,dv\right\}\quad \mbox{pre-compact in $L^p_{loc}(\mathbb{R}^{d+1}_+)$},
\]
for any $p\in[1,+\infty).$
Thus, we can select a subsequence (still labeled by $h$) such that for some $f\in L^\infty(\mathbb{R}^{d+1}\times[0,M]),$ with values in $[0,1],$ for which 
\[
f^h\to f\quad \mbox{*--weakly in $L^\infty(\mathbb{R}^{d+1}\times[0,M])$};
\]
\[
\int f^h\,dv,\,\int v f^h\,dv,\,\int \eta_\e f^h\,dv\to
\int f\,dv, \,\int vf\,dv,\,\int \eta_\e f\,dv,
\]
a.e. $(x,t)$ and in $L^p_{loc}(\mathbb{R}^{d+1}_+);$

\[
\Pi^{eq}_{f^h}\to
\Pi^{eq}_f
\]
a.e. $(x,t)$ and in $L^p_{loc}(\mathbb{R}^{d+1}_+).$

Estimates \eqref{E:3}, \eqref{E:4} imply that
\begin{align*}
f(x,t,v){}={}0, &  \quad (x,t,v)\in\mathbb{R}^{d+1}_+\times(\rho(x,t)+\e,M),\\
f(x,t,v){}={}1, &  \quad (x,t,v)\in\mathbb{R}^{d+1}_+\times(0,\max\{0,\rho(x,t)-\e\}).
\end{align*}

This implies that a.e. $(x,t),$
\[
\int v(f-\Pi^{eq}_f)\,dv\leq \int_{\max\{0, \rho-\e\}}^{\rho+\e}v\,dv\leq 4\e\int f\,dv,
\]
which establishes  part 2 of the theorem.

Similarly, for any $i,$
\begin{equation}
\label{cond:A'}
\left| \int A'_i(v)(f-\Pi^{eq}_f)\,dv\right|\leq C\e\rho,
\end{equation}
for some $C$ determined by $A_i.$ This establishes the equation \eqref{part:1}.

It remains to show that there is a measure $\mu_{x,t}$ with the properties stated in the theorem. We follow the approach from \cite{M1}, where a similar fact is established.

Let $a_i(v)$ be a continuously differentiable extension of $A_i'(v)$ restricted to the interval $(\max\{0,\rho-\e\},\rho+\e):$
\begin{align*}
a_i(v){}={}A_i'(v), & \quad v\in (\max\{0,\rho-\e\},\rho+\e)\\
a_i(v){}={}0, & \quad v\in (\max\{0,\rho-2\e\},\rho+2\e)^c.
\end{align*}
Functions $a_i$ depend on $(x,t)$ through $\rho=\rho(x,t),$ which we implicitely assume.

Condition \eqref{main:nondegeneracy} implies that set $\{a_i\}$ is linearly independent on $[0,M].$ Let $f_0$ be the projection of $f-\Pi^{eq}_f$ to $\mbox{\rm Span}\{a_1,...,a_d\}\subset L^2((0,M)).$ Thus,
\[
f_0(v){}={}\sum_{i=1}^d\alpha_ia_i(v),
\]
and due to estimate \eqref{cond:A'}, all $|\alpha_i|\leq C\e,$
for some $C>0,$ independently of $(x,t).$ Note, that also 
\[
\mbox{\rm diam}(\supp f_0)\leq 4\e.
\]
The measure $\mu_{x,t}$ can be defined as
\[
\mu_{x,t}{}={}f_0'(v)\,dv{}+{}\delta(v-\rho(x,t)).
\]

\end{proof}

\end{section}

\end{document}